\documentclass[reqno,11pt]{amsart}

\usepackage{amsmath,amssymb,amsthm,mathrsfs}
\usepackage{epsfig,color}
\usepackage[latin1]{inputenc}

\usepackage[numbers]{natbib}

\numberwithin{equation}{section}

\def\A{\mathcal A}

\def\H{\mathcal H}

\def\MM{\mathbf M}
\def\Le{\mathcal L}
\def\R{\mathbb R}
\def\N{\mathbb N}
\def\Z{\mathbb Z}
\newcommand{\dist}{\mathop{\mathrm{dist}}}
\def\e{\varepsilon}
\def\s{\sigma}
\def\S{\Sigma}
\def\vphi{\varphi}

\def\om{\omega}

\def\g{\gamma}

\def\Om{\Omega}
\def\de{\delta}
\def\Id{{\rm Id}}

\def\spt{{\rm spt}}

\def\pa{\partial}

\def\00{{\bf 0}}

\def\cl{{\rm cl}\,}

\def\F{\mathcal{F}}

\def\W{\mathcal{W}}
\def\P{\mathcal{P}}

\def\CC{\mathcal{C}}

\renewcommand{\a}{\alpha}

\renewcommand{\om}{\omega}

\newcommand{\cc}{\subset\subset}

\def\Lip{{\rm Lip}\,}

\newcommand\res{\mathop{\hbox{\vrule height 7pt width .3pt depth 0pt \vrule height .3pt width 5pt depth 0pt}}\nolimits}

\def\weak{\stackrel{*}{\rightharpoonup}}

\def\MM{\mathbf{M}}

\def\Q{\mathcal{Q}}

\newtheorem*{theorem*}{Theorem}
\newtheorem{theorem}{Theorem}
\newtheorem{lemma}[theorem]{Lemma}

\newtheorem{remark}[theorem]{Remark}

\newtheorem{definition}[theorem]{Definition}

\title{A direct approach to Plateau's problem}

\author{C. De Lellis}
\address{Institut f\"ur Mathematik, Universitaet Z\"urich, Winterthurerstrasse 190, CH-8057 Z\"urich, Switzerland}
\email{camillo.delellis@math.uzh.ch}

\author{F. Ghiraldin}
\address{Institut f\"ur Mathematik, Universitaet Z\"urich, Winterthurerstrasse 190, CH-8057 Z\"urich, Switzerland}
\email{francesco.ghiraldin@math.uzh.ch}

\author{F. Maggi}
\address{Department of Mathematics, The University of Texas at Austin,  2515 Speedway Stop C1200, Austin, Texas 78712-1202, USA}
\email{maggi@math.utexas.edu}

\begin{document}

\begin{abstract} We provide a compactness principle which is applicable
to different formulations of Plateau's problem in codimension one and which
is exclusively based
on the theory of Radon measures and elementary comparison arguments.
Exploiting some additional techniques in geometric
measure theory, we can use this principle to give a different proof of a
theorem by Harrison and Pugh and to answer a question raised by Guy
David.
\end{abstract}

\maketitle

\section{Introduction}

Since the pioneering work of Reifenberg there has been an ongoing interest into formulations of Plateau's problem involving the minimization of the Hausdorff measure on closed sets coupled with some notion of ``spanning a given boundary''. More precisely consider any closed set $H\subset \mathbb R^{n+1}$ and assume to have a class $\mathcal{P} (H)$ of relatively closed subsets $K$ of $\mathbb R^{n+1}\setminus H$, which encodes a particular notion of ``$K$ bounds $H$''. Correspondingly there is a formulation of Plateau's problem, namely the minimum for such problem is
\begin{equation}
  \label{plateau problem generale}
m_0 :=  \inf \{\H^n (K) : K\in \mathcal{P}(H)\}\,,
\end{equation}
and a minimizing sequence $\{K_j\} \subset \mathcal{P} (H)$ is characterized by the property $\mathcal{H}^n (K_j) \to m_0$.
Two good motivations for considering this kind of approach rather than the one based on integer rectifiable currents are that, first, not every interesting boundary can be realized as an integer rectifiable cycle and, second, area minimizing $2$-d currents in $\mathbb R^3$ are always smooth away from their boundaries, in contrast to what one observes with real world soap films.

There are substantial difficulties related to the minimization of Hausdorff measures on classes of closed (or even compact) sets. Depending on the convergence adopted, these are either related to lack of lower semicontinuity or to compactness issues. In both cases, obtaining existence results in this framework is a quite delicate task, as exemplified in various works by Reifenberg \cite{reifenberg1,reifenberg2,reifenberg3}, De Pauw \cite{depauwPlateau}, Feuvrier \cite{Feuvrier2009}, Harrison and Pugh \cite{Harrison2011,Harrison2014,HarrisonPugh14}, Fang \cite{Fang2013} and  David \cite{davidplateau}.

Our goal here is to show that in some interesting cases these difficulties can be avoided by exploiting Preiss' rectifiability theorem for Radon measures \cite{Preiss,DeLellisNOTES} in combination with the sharp isoperimetric inequality on the sphere and with standard variational arguments, noticeably elementary comparisons with spheres and cones. A precise formulation of our main result is the following:

\begin{definition}[Cone and cup competitors]\label{def good class}
Let $H\subset \mathbb R^{n+1}$ be closed. Given $K\subset \mathbb R^{n+1}\setminus H$ and $B_{x,r}=\{y\in\R^n:|x-y|<r\}\subset \mathbb R^{n+1}\setminus H$, the cone competitor for $K$ in $B_{x,r}$ is the set
\begin{equation}
  \label{cone comp}
\big(K\setminus B_{x,r}\big) \cup \big\{\lambda x+(1-\lambda)z:z\in K\cap\pa B_{x,r}\,,\lambda\in[0,1]\big\}\,;
\end{equation}
a cup competitor for $K$ in $B_{x,r}$ is any set of the form
\begin{equation}
  \label{cup comp}
  \big(K\setminus B_{x,r}\big) \cup \big(\partial B_{x,r}\setminus A\big)\,,
\end{equation}
where $A$ is a connected component of $\partial B_{x,r} \setminus K$.

Given a family $\mathcal{P} (H)$ of relatively closed subsets $K\subset \mathbb R^{n+1}\setminus H$, we say that an element $K\in \mathcal{P} (H)$ has the {\em good comparison property} in $B_{x,r}$ if
\begin{equation}
  \label{inf good class}
  \inf \big\{ \H^n (J) : J\in \mathcal{P} (H)\,,J\setminus \cl (B_{x,r}) =K\setminus \cl
(B_{x,r})\big\} \leq \H^n (L)\,
\end{equation}
whenever $L$ is the cone competitor or any cup competitor for $K$ in $B_{x,r}$.
The family $\mathcal{P} (H)$ is a {\em good class} if, for any $K\in \mathcal{P} (H)$ and for every $x\in K$, the set $K$ has the good comparison property in $B_{x,r}$ for a.e. $r\in (0, \dist (x, H))$.
\end{definition}

\begin{theorem}\label{thm generale}
Let $H\subset \mathbb R^{n+1}$ be closed and $\mathcal{P} (H)$ be a good class. Assume the infimum in Plateau's problem \eqref{plateau problem generale} is finite and let $\{K_j\}\subset \mathcal{P}(H)$ be a minimizing sequence of countably $\mathcal{H}^n$-rectifiable sets. Then, up to subsequences, the measures $\mu_j := \H^n \res K_j$ converge weakly$^\star$ in $\mathbb R^{n+1}\setminus H$ to a measure $\mu = \theta \H^n \res K$, where $K = \spt\, \mu \setminus H$ is a countably $\H^n$-rectifiable set and $\theta \geq 1$. In particular, $\liminf_j\H^n(K_j)\ge \H^n(K)$.

Furthermore, for every $x\in K$ the quantity $r^{-n} \mu (B_{x,r})$ is monotone increasing and
\begin{equation}\label{e:DLB}
\theta (x) = \lim_{r\downarrow 0} \frac{\mu (B_\rho (x))}{\omega_n \rho^n} \geq 1\, ,
\end{equation}
where $\omega_n$ is the Lebesgue measure of the unit ball in $\mathbb R^n$.
\end{theorem}


Our point is that although Theorem \ref{thm generale} does not imply in general the existence of a minimizer in $\mathcal{P} (H)$, this might be achieved with little additional work (but possibly using some heavier machinery from geometric measure theory) in some interesting cases. We will give here two applications. The first one is motivated by a very elegant idea of Harrison, which can be explained as follows. Assume that $H$ is a smooth closed compact $n-1$-dimensional submanifold of $\mathbb R^{n+1}$: then we say that a relatively closed set $K\subset \mathbb R^{n+1}\setminus H$ bounds $H$ if $K$ intersects every smooth curve $\gamma$ whose linking number with $H$ is $1$. A possible formulation of Plateau's problem is then to minimize the Hausdorff measure in this class of sets. Building upon her previous work on differential chains, see \cite{HarrisonOperator12}, in \cite{Harrison2011} Harrison gives a general existence result for a suitable weak version of this problem. In the subsequent work \cite{HarrisonPugh14}, Harrison and Pugh prove that the corresponding minimizer
yields a closed set $K$ which is a minimizer in the original formulation of the problem, and to which the regularity theory for $(\MM,\xi,\de)$-minimal sets by Almgren and Taylor \cite{Almgren76,taylor76} can be applied. In particular, $K$ is analytic out of a $\H^n$-negligible singular set, and, actually, in the physical case $n=3$ and away from the boundary set $H$, this singular set obeys the experimental observations known as Plateau's laws. Boundary regularity seems a major issue to be settled.

We can recover the theorem of Harrison and Pugh in a relatively short way from Theorem \ref{thm generale}. In fact our approach allows one to work, with the same effort, in a more general setting.

\begin{definition}\label{def plateau first}
Let $n\ge 2$ and $H$ be a closed set in $\R^{n+1}$. When $H$ is a closed compact $n-1$-dimensional submanifold, following \cite{HarrisonPugh14} we say that a closed set $K\subset \mathbb R^{n+1}\setminus H$ spans $H$ if it intersects any smooth embedded closed curve $\gamma$ in $\mathbb R^{n+1}\setminus H$ such that the linking number of $H$ and $\gamma$ is $1$.

More in general, for an arbitrary closed $H$ let us consider the family
\[
\CC_H=\big\{\g:S^1\to\R^{n+1}\setminus H:\mbox{$\g$ is a smooth embedding of $S^1$ into $\R^{n+1}$}\big\}\,.
\]
We say that $\CC\subset\CC_H$ is closed by homotopy (with respect to $H$) if $\CC$ contains all elements
$\gamma'\in \CC_H$ belonging to the same homotopy class $[\g] \in\pi_1(\R^{n+1}\setminus H)$ of any $\gamma \in \CC$. Given $\CC\subset\CC_H$ closed by homotopy, we say that a relatively closed subset $K$ of $\R^{n+1}\setminus H$ is a $\CC$-spanning set of $H$ if
\begin{eqnarray}\label{Cfilling}
\mbox{$K\cap\g\ne\emptyset$ for every $\g\in\CC$}\,.
\end{eqnarray}
We denote by $\mathcal{F} (H, \CC)$ the  family of $\CC$-spanning sets of $H$.
\end{definition}

\begin{theorem}\label{thm plateau}
Let $n\ge 2$, $H$ be closed in $\R^{n+1}$ and $\CC$ be closed by homotopy with respect to $H$. Assume the infimum of the Plateau's problem corresponding to $\P(H)=\F(H,\CC)$ is finite. Then:
\begin{itemize}
\item[(a)]  $\F (H,\CC)$ is a good class in the sense of Definition \ref{def good class}.
\item[(b)] There is a minimizing sequence $\{K_j\}\subset \F (H, \CC)$ which consists of $\H^n$-rectifiable sets. If $K$ is any set associated to $\{K_j\}$ by Theorem \ref{thm generale}, then $K\in\F (H, \CC)$ and thus $K$ is a minimizer.
\item[(c)] The set $K$ in (b) is an $(\MM , 0, \infty)$-minimal set in $\R^{n+1}\setminus H$ in the sense of Almgren.
\end{itemize}
\end{theorem}

\begin{remark}\label{r:Harr}
{\rm As already mentioned the variational problem considered in \cite{Harrison2011,HarrisonPugh14} corresponds to the case where $H$ is a closed compact $(n-1)$-dimensional submanifold of $\R^{n+1}$ and $\CC =\{\gamma\in \CC_H: \mbox{ the linking number of $H$ and $\gamma$ is $1$}\}$. In fact there is yet a small technical difference: in \cite{Harrison2011,HarrisonPugh14} the authors minimize the Hausdorff spherical measure, which coincides with the Hausdorff measure $\mathcal{H}^n$ on rectifiable sets, but it is in general larger on unrectifiable sets.
After completing this note we learned that Harrison and Pugh have been able to improve their proof in order to minimize as well the Hausdorff measure, \cite{Personal}. Finally, we stress that, while points (a) and (c) can be concluded from Theorem \ref{thm generale} using elementary results about Radon measures and isoperimetry, point (b) relies in a substantial way upon the theory of Caccioppoli sets and minimal partitions.}
\end{remark}

We next exploit Theorem \ref{thm generale} in a second context proving an existence result for the ``sliding minimizers''  introduced by David, see \cite{davidplateau,DavidBeginners}.

\begin{definition}
  Let $H\subset \R^{n+1}$ be closed and $K_0 \subset \R^{n+1}\setminus H$ be relatively closed. We denote by $\S(H)$ the family of Lipschitz maps $\vphi:\R^{n+1}\to\R^{n+1}$ such that there exists a continuous map $\Phi:[0,1]\times\R^{n+1}\to\R^{n+1}$ with $\Phi(1,\cdot)=\vphi$, $\Phi(0,\cdot)=\Id$ and $\Phi(t,H)\subset H$ for every $t\in[0,1]$. We then define
  \[
  \A(H,K_0)=\big\{K:\mbox{$K=\vphi(K_0)$ for some $\vphi\in\S(H)$}\big\}\,
  \]
and say that $K_0$ is a sliding minimizer if $\H^n(K_0)=\inf\{\H^n(J):J\in\A(H,K_0)\}$.
\end{definition}

We will use the convention that, whenever $E\subset \mathbb R^{n+1}$ and $\delta>0$, $U_\delta (E)$ denotes the $\delta$-neighborhood of $E$.

\begin{theorem}\label{thm david}
$\A(H,K_0)$ is a good class in the sense of Definition \ref{def good class}.
Moreover, assume that
\begin{itemize}
\item[(i)] $K_0$ is bounded and countably $\H^n$-rectifiable with $\H^n(K_0)<\infty$;
\item[(ii)] $\H^n(H)=0$ and for every $\eta>0$ there exist $\de>0$ and $\pi\in\S(H)$ such that
  \begin{equation}
    \label{retraction}
      \Lip\pi\le1+\eta\,,\qquad \pi(U_\de(H))\subset H\,.
  \end{equation}
\end{itemize}
Then, given any minimizing sequence $\{K_j\}$ (in the Plateau's problem corresponding to $\P(H)=\A(H,K_0)$) and any set $K$ as in Theorem \ref{thm generale}, we have
  \begin{equation}
    \label{allardo}
     \inf \big\{\H^n (J) : J\in \A (H, K_0)\big\} =\H^n(K)=\inf\big\{\H^n(J):J\in\A(H,K)\big\}\,.
  \end{equation}
In particular $K$ is a sliding minimizer.
\end{theorem}

The proof of the second equality in \eqref{allardo} borrows important ideas from the work of DePauw and Hardt, see \cite{depauwhardt} and it uses in a substantial way the theory of varifolds, in particular Allard's regularity theorem. A different approach to the existence of a $K$ satisfying the left hand side of \eqref{allardo} has been suggested by David in Section 7 of \cite{davidplateau}, where he also raised the question whether one could conclude the equality on the right hand side. Our Theorem gives therefore a positive answer to this question (see below for a stronger one raised also by David).

\begin{remark}\label{r:2}
  {\rm It seems very hard to conclude something about the existence of a minimizer in the {\it original} class $\A(H,K_0)$ from our approach, without a deeper analysis of what sliding deformations can do to the starting set $K_0$. The following example illustrates this difficulty. Let $H$ be the union of two far away parallel circles and $K_0$ be a cylinder joining them, namely define,
  for $R$ large,
  \begin{align*}
  H &= \{(x_1, x_2, x_3)\in \mathbb R^3: x_1^2 + x_2^2 = 1, |x_3| = R\}\\
  K_0 &= \{(x_1, x_2, x_3)\in \mathbb R^3: x_1^2 + x_2^2 = 1, |x_3| < R\}\,.
  \end{align*}
  Let $\{K_j\}\subset \A (H, K_0)$ be a minimizing sequence and $\mu_j = \H^2 \res K_j$. We obviously expect that $\H^2\res K_j \to \H^2 \res K$ where
  \[
  K =\{(x_1, x_2, x_3): x_1^2 + x_2^2 < 1, |x_3| = R\}\,.
  \]
Of course $K\not\in\A(H,K_0)$, but we can easily build a map $\varphi \in \S (H)$ which ``squeezes'' $K_0$ onto the set $K_1
= K \cup \{(0,0,t):|t|\le R\}$, i.e. the top and bottom disks connected by a vertical segment. $K_1$ is then a minimizer in $\A(H,K_0)$. On the other hand $K = \spt (\H^2\res K_1)$ and thus a purely measure-theoretic approach does not seem to capture this phenomenon. It is however very tempting to conjecture that, upon adding a suitable $\mathcal{H}^n$-negligible set (and possibly some more requirements on the boundary $H$), any
set $K$ as in Theorem \ref{thm david} is an element of $\mathcal{A} (H, K_0)$; cf. \cite{davidplateau}. We refer the reader to \cite{BW} for a result which has a similar flavour.}
\end{remark}

%

\noindent {\bf Acknowledgement}: CDL has been supported by ERC 306247 {\it Regularity of area-minimizing currents} and by
SNF 146349 {\it Calculus of variations and fluid dynamics}. FG has been supported by SNF 146349. FM has been supported by the NSF Grant DMS-1265910 {\it Stability, regularity, and symmetry issues in geometric variational problems}, and by a Simons visiting professorship of the Mathematisches Forschungsinstitut Oberwolfach. The authors thank Guy David and Guido De Philippis for many interesting comments and discussions.

\section{Proof of Theorem \ref{thm generale}}

We start with following classical fact. We include a quick proof just for the reader's convenience using sets of finite perimeter; the latter are however not really necessary, in particular it should be possible to prove Theorem \ref{thm generale} without leaving the framework provided by the theory of Radon measures.
In what follows we use the notation $\s_k=\H^{k}(\{z\in\R^{k+1}:|z|=1\})$ and $\omega_{k+1} = \H^{k+1} (\{z\in \R^{k+1} : |z|\leq 1\}) = \frac{\sigma_k}{k+1}$.

\begin{lemma}[Isoperimetry on the sphere]\label{lemma iso sfera}
  If $J\subset\pa B_{x,r}$ is compact and $\{A_h\}_{h=0}^\infty$ is the family of the connected components of $\pa B_{x,r}\setminus J$, ordered so that $\H^n(A_h)\ge\H^n(A_{h+1})$, then
\begin{equation}\label{iso sfera}
  \H^n(\pa B_{x,r}\setminus A_0)\le C(n)\,\H^{n-1}(J)^{n/n-1}\,.
\end{equation}
Moreover, for every $\eta>0$ there exists $\de>0$ such that
\begin{equation}\label{iso ottimale}
\min\big\{\H^n(A_0),\H^n(A_1)\big\}\ge\left(\textstyle{\frac{\s_n}{2}}-\de\right)\,r^n\qquad\Rightarrow\qquad \H^{n-1}(J)\ge(\s_{n-1}-\eta)\,r^{n-1}\, .
\end{equation}

The inequality \eqref{iso sfera} holds also if we replace $\partial B_{x,r}$ with $\partial Q$ for any cube $Q\subset \mathbb R^{n+1}$ or with any spherical cap $\partial B_{x,r} \cap \{y: (y-x) \cdot \nu > \varepsilon r\}$, where $\nu \in S^n$ and $\varepsilon \in ]0,1[$.
\end{lemma}

\begin{proof}
  [Proof of Lemma \ref{lemma iso sfera}] We first prove \eqref{iso sfera} with $J\subset \partial B_{x,r}$. The proof can be easily adapted to boundary of cubes and spherical caps.
Since $\partial A_h \subset J$ and (without loss of generality) $\H^{n-1}(J)<\infty$ we know that \cite[Prop. 3.62]{AFP} each $A_h$ has finite perimeter
 and $\partial^*A_h \subset J$ (where $\partial^* A_h$ denotes the reduced boundary). By the properties of the reduced boundary one easily infers that
 $\sum_h \H^{n-1}\res \partial^*A_h \leq 2\H^{n-1}\res J$.
 By the relative isoperimetric inequality on $\pa B_{x,r}$, if $A\subset \pa B_{x,r}$ is of finite perimeter, then
 \begin{equation}\label{eq:isopsfera}
 \min\Big\{ \H^n( A)  ,\H^n(\pa B_{x,r}\setminus A) \Big\}  \le C(n)\,\H^{n-1}(\partial^*A)^{n/n-1}\,.
 \end{equation}
By the ordering property of the $\H^n(A_h)$, we thus find
\[
 \H^n( A_h) \le C(n)\,\big[\H^{n-1}(\partial^*A_h)\big]^{n/n-1}\,,\qquad\forall h\ge 1\,.
\]
Adding up over $h\geq 1$, the superadditivity of the function $t\mapsto t^\frac{n}{n-1}$ yields
\[
\H^n( \pa B_{x,r}\setminus A_0) \le C(n)\bigg(\sum_{h\ge 1}\H^{n-1}(\partial^*A_h)\bigg)^{n/n-1}\leq C(n)\,\H^{n-1}(J)^{n/n-1}\,.
\]
\eqref{iso ottimale} can be proved via a compactness argument: assuming that it fails for a given $\eta>0$, we find a sequence $J_k$ of sets, each violating the statement for $\delta = \frac{1}{k}$. Letting $A^k_0$ and $A^k_1$ be the corresponding connected components, we can use the compactness of Cacciopoli sets to conclude that they are converging to two sets $A_0^\infty, A_1^\infty$ with
\begin{eqnarray}\label{pino1}
\H^n(A_0^\infty)=\H^n(A_1^\infty)=\frac{\s_n}2\,r^n\,,\qquad \H^n(A_0^\infty\cap A_1^\infty)=0,
\\\label{pino2}
\max\Big\{\H^{n-1}(\pa^* A_0^\infty ),\H^{n-1}(\pa^* A_1^\infty )\Big\}\le(\s_{n-1}-\eta)\,r^{n-1}\, ,
\end{eqnarray}
By \eqref{pino1}, $\pa^* A_0^\infty = \pa^* A_1^\infty$; but then \eqref{pino2}  contradicts the sharp isoperimetric inequality on the sphere \cite[Theorem 10.2.1]{buragozalgaller}.
\end{proof}

\begin{proof}
  [Proof of Theorem \ref{thm generale}] Up to extracting subsequences we can assume the existence of a Radon measure $\mu$ on $\R^{n+1}\setminus H$ such that
\begin{equation}\label{muj va a mu}
  \mu_j\weak\mu\,,\qquad\mbox{as Radon measures on $\R^{n+1}\setminus H$}\,,
\end{equation}
where $\mu_j=\H^n \res K_j$. We set $K = \spt\, \mu \setminus H$ and divide the argument in four steps.

\medskip

\noindent {\it Step one}: We show the existence of $\theta_0=\theta_0(n)>0$ such that
\begin{equation}
  \label{lower density estimate mu}
  \mu(B_{x,r})\ge\theta_0\,\omega_n r^n\,,\qquad\forall x\in\spt\,\mu\,,\,\forall r<d_x :=\dist(x,H) \, .
\end{equation}
By \cite[Theorem 6.9]{mattila}, \eqref{lower density estimate mu} implies
\begin{eqnarray}
  \label{lower bound on mu}
  \mu\ge{\theta_0}\,\H^n\res K\,,\qquad\mbox{on subsets of $\R^{n+1}\setminus H$}\,.
\end{eqnarray}
We now prove \eqref{lower density estimate mu}.
Let $f(r)=\mu(B_{x,r})$ and $f_j(r)=\H^n(K_j\cap B_{x,r})$, so that
\[
f_j(r)-f_j(s)\ge\int_s^r\,\H^{n-1}(K_j\cap\pa B_{x,t})\,dt\,,\qquad 0<s<r<d_x\,,
\]
by the coarea formula \cite[3.2.22]{FedererBOOK}. Since $f_j$ is increasing on $(0,d_x)$, one has,
\[
Df_j\ge f_j'\,\Le^1\,,\qquad\mbox{with $f_j'(r)\ge \H^{n-1}(K_j\cap\pa B_{x,r})$ for a.e. $r\in(0,d_x)$}\,
\]
(here $Df_j$ denotes the distributional derivative of $f_j$, $f_j'$ the pointwise derivative and $\Le^1$ the Lebesgue measure).
By Fatou's lemma, if we set $g(t)=\liminf_{j}f_j'(t)$, then
\[
f (r) - f(s) = \mu(B_{x,r}\setminus B_{x,s})\ge\int_s^r\,g(t)\,dt\,,\qquad\mbox{provided $\mu(\pa B_{x,r})=\mu(\pa B_{x,s})=0$}\,.
\]
This shows that $Df \geq g \Le^1$. On the other hand, using the differentiability a.e. of $f$ and letting $s\uparrow r$, we also conclude
$f'\geq g$ $\Le^1$-a.e., whereas $Df \geq f' \Le^1$ is a simple consequence of the fact that $f$ is an increasing function.

Let $A_j$ denote a connected component of $\pa B_{x,r}\setminus K_j$ of maximal $\H^n$-measure, and let $K''_j$ be the corresponding cup competitor of $K_j$ in $B_{x,r}$, see \eqref{cup comp}. Since $\P(H)$ is a good class, for a.e. $r<d_x$ by \eqref{iso sfera} we find
\begin{equation}\label{eq:stimacomparison}
f_j(r)\le \H^n(\pa B_{x,r}\setminus A_j)+\e_j\le C(n)\,\Big(\H^{n-1}(\pa B_{x,r}\cap K_j)\Big)^{n/(n-1)}+\e_j\,,
\end{equation}
where $\e_j\to 0$ takes into account the almost minimality of $K_j$, namely we assume $\H^n (K_j) \leq \inf \{\H^n (K): K\in \mathcal{P} (H)\} + \e_j$.
Letting $j\to\infty$ we find that
\[
f(r)\le C(n)\,g(r)^{n/(n-1)}\le C(n)\,f'(r)^{n/(n-1)}\,,\qquad\mbox{ for a.e. $r<d_x$}\,,
\]
from which
\[
 f(r)^{(n-1)/n}\le C(n)\,f'(r)\,,\qquad\mbox{ for a.e. $r<d_x$}\,,
\]
which implies
\[
 1\le C(n) \big(f(r)^{1/n}\big)'\,,\qquad\forall r<d_x\,.
\]
Since the distributional derivative $Df^{1/n}$ is nonnegative, we deduce $r\le C(n)( f(r)^{1/n}-f(0)^{1/n})$, hence $\mu(B_{x,r})\ge \theta_0 \omega_n r^n$ for a suitable value of $\theta_0$.

\medskip

\noindent {\it Step two}: We fix $x\in \spt\, \mu \setminus H$, and prove that
\begin{equation}\label{monotonicity mu}
 r\mapsto \frac{f(r)}{r^n}=\frac{\mu(B_{x,r})}{r^n}\quad\mbox{is increasing on $(0,d_x)$.}
\end{equation}
This property can be deduced by using the cone competitor in $B_{x,r}$ in place of the cup competitor: estimate \eqref{eq:stimacomparison} becomes now
\[
 f_j(r)\le \H^n(K'_j\cap B_{x,r}) +\e_j= \frac rn \H^{n-1}(K_j\cap\pa B_{x,r}) +\e_j\le  \frac rn f_j'(r)+\e_j\,,
 \]
yielding $f(r)\le \frac rn g(r)\le \frac rn f'(r)$ for a.e. $r<d_x$. Again the positivity of the measure $D\log(f)$ implies the claimed monotonicity formula. By \eqref{lower bound on mu} and \eqref{monotonicity mu} the $n$-dimensional density of the measure  $\mu$, namely:
\[
\theta(x)=\lim_{r\to 0^+}\frac{f(r)}{\omega_n r^n}\ge{\theta_0}\,.
\]
exists, is finite and positive $\mu$-almost everywhere. By the well known theorem of Preiss, cf. \cite[Theorem 1.1]{DeLellisNOTES}, this property implies that $\mu = \theta \H^n \res \tilde{K}$ for some countably $\H^n$-rectifiable set $K$ and some positive Borel function $\theta$.
Since $K$ is the support of $\mu$, $\H^n (\tilde{K}\setminus K) =0$. On the other hand $\H^n (K\setminus \tilde{K}) =0$ by \eqref{lower bound on mu} and thus $K$ must be rectifiiable and $\mu = \theta \H^n\res K$. 

\medskip

\noindent {\it Step three}: We prove that $\theta(x)\ge 1$ for every $x\in K$ such that the approximate tangent space
to $K$ exists (thus, $\H^n$-a.e. on $K$). Fix any such $x\in K\setminus H$ and suppose, up to rotating the coordinates, that $T= \{x_{n+1}=0\}$ is the approximate tangent space to $K$ at $x$: in particular (cf. \cite[Corollary 4.4]{DeLellisNOTES}),
\[
 \H^n\res \frac{K-x}{r} \rightharpoonup^* \H^n\res T\,,\qquad\mbox{as $r\to 0^+$}\,.
\]
By the density lower bound \eqref{lower density estimate mu}, for every $\e>0$ there is $\rho>0$ such that
\begin{equation}\label{e:uniforme}
 K\cap B_{x,r}  \subset x+\{|y_{n+1}|< \e\,r\} \qquad \forall r<\rho\, .
\end{equation}
Indeed, assume $r$ is sufficiently small so that $\mu (B_{x,2r}\setminus (x+\{|y_{n+1}|< \frac{\e}{2}\,r\}))< \theta_0 2^{-n} \e^n r^n$. Then $K\cap (x+\{|y_{n+1}|< \frac{\e}{2}\,r\}) \cap B_{x,r}$ must be empty, since the existence of a point belonging to that set would imply
\[
\mu (B_{x,2r}\setminus (x+\{|y_{n+1}|< \textstyle{\frac{\e}{2}}\,r\})) \geq \mu (B_{y, \e r/2})\geq \frac{\theta_0 \e^n r^n}{2^n}\, .
\]
Setting $c(\e)=\e/\sqrt{1-\e^2}$, \eqref{e:uniforme} can be equivalently stated as
\begin{equation}\label{e:uniforme2}
 K\cap B_{x,\rho}\subset x+\{(y',y_{n+1}):|y_{n+1}|<c(\e)\,|y'|\}\,.
\end{equation}
If require in addition that $\H^n(K\cap\pa B_{x,\rho})=0$, then by the coarea formula \cite[3.2.22]{FedererBOOK}
\begin{eqnarray*}
  0&=&\lim_{j\to\infty}\mu_j\big(\cl(B_{x,\rho})\cap \big(x+\{(y',y_{n+1}):|y_{n+1}|<c(\e)\,|y'|\}\big)\big)
  \\
  &\ge&\int_0^\rho\,\H^{n-1}\big(K_j\cap\pa B_{x,r}\cap \big(x+\{(y',y_{n+1}):|y_{n+1}|<c(\e)\,|y'|\}\big)\,dr\, .
\end{eqnarray*}
So, if $\pa B^+_{x,r,\e} :=\{y\in \pa B_{x,r}:y_{n+1} >x_{n+1}+\e r\}$ and $\pa B^-_{x,r,\e} :=\{y\in \pa B_{x,r}:y_{n+1}<x_{n+1}-\e r\}$,
\begin{equation}
  \label{rhorho}
  \liminf_{j\to\infty}\H^{n-1}(K_j\cap\pa B_{x,r,\e}^\pm)=0\,,\qquad\mbox{for a.e. $r<\rho$}\,.
\end{equation}
Let us fix $r<\rho$ such that \eqref{rhorho} holds, $f'(r)$ exists, $f'(r)\ge g(r)$, and each $K_j$ has the good comparison property in $B_{x,r}$ (all these
conditions can be ensured for a.e. $r$). Using Lemma \ref{lemma iso sfera}, namely the relative isoperimetric inequality in the spherical cap $\pa B_{x,r,\e}^+$, one finds that if $A_j^+$ denotes the connected component of $\pa B_{x,r,\e}^+$ with largest $\H^n$-measure, then $\H^n(\pa B_{x,r,\e}^+\setminus A_j^+)\le C(n)\,\H^{n-1}(K_j\cap\pa B_{x,r,\e}^+)$, and thus, by \eqref{rhorho}, that
\[
 \lim_{j\rightarrow\infty} \H^n(A^+_j)=\H^n(\pa B^+_{x,\e,r})\,;
\]
similarly, $\H^n(A_j^-)\to\H^n(\pa B^-_{x,\e,r})$ if $A_j^-$ is the largest connected component of $\pa B_{x,r,\e}^-\setminus K_j$. We claim that, for $j$ sufficiently large, $A^+_j$ and $A^-_j$ cannot belong to the same connected component of $\pa B_{x,r}\setminus K_j$:
for otherwise, we can compare with the cup competitor of $K_j$ in $B_{x,r}$ defined by the connected component of $\pa B_{x,r}\setminus K_j$ containing $A^+_j\cup A^-_j$ (which is the largest connected component of $\pa B_{x,r}\setminus K_j$ when $j$ if large enough), obtaining
\begin{eqnarray*}
 \mu(B_{x,r})&\le&\liminf_{j\rightarrow\infty} \H^n(K_j\cap B_{x,r}) \leq
 \liminf_{j\rightarrow\infty} \H^n(\pa B_{x,r}\setminus (A^+_j\cup A^-_j))
 \\
 &\le&\H^n(\pa B_{x,r}\cap\{|y_{n+1}-x_{n+1}|<\e\,r\})\le C\e r^n,
\end{eqnarray*}
against the density lower bound \eqref{lower density estimate mu}. If we now fix $\eta$, then we can choose $\e$ so that
 Lemma \ref{lemma iso sfera} entails, for $j$ large enough,
\[
( \s_{n-1} -\eta)r^{n-1} \le \liminf_{j\rightarrow\infty} \H^{n-1}(K_j\cap \pa B_{x,r})\le f'(r)\,.
\]
In conclusion, $f'(r)\ge( \s_{n-1} -\eta)r^{n-1}$ for a.e. $r<\rho$. Inasmuch $f(r)\ge (\s_{n-1} - \eta)r^n/n$ for every $r<\rho$, one concludes that $\theta (x) \ge (\s_{n-1} - \eta)/(n\omega_n)$. Letting $\eta\rightarrow 0$ we obtain $\theta(x)\ge 1$.

To complete the proof of the theorem we recall that, a standard consequence of the monotonicity formula \eqref{monotonicity mu} is the upper semicontinuity of $\theta$: a simple density argument then shows \eqref{e:DLB} (cf. \cite[Corollary 17.8]{SimonLN}).
\end{proof}

\section{Proof of Theorem \ref{thm plateau}}

Most of the proof of Theorem \ref{thm plateau} relies on the following elementary geometric remark.

\begin{lemma}\label{lemma curve intersezione}
   If $K\in\F(H,\CC)$, $B_{x,r}\cc\R^{n+1}\setminus H$, and $\g\in\CC$,
   then either $\g\cap (K\setminus B_{x,r})\ne\emptyset$, or there exists a connected component $\s$ of
   $\g\cap \cl (B_{x,r})$ which is homeomorphic to an interval and whose end-points belong to two distinct
   connected components of $\cl(B_{x,r})\setminus K$ (and so to two distinct components of $\pa B_{x,r}\setminus K$). The same conclusion holds if
we replace $B_{x.r}$ with an open cube $Q\subset \mathbb R^{n+1}\setminus H$.
\end{lemma}

\begin{proof}
  [Proof of Lemma \ref{lemma curve intersezione}] {\it Step one}: We first prove the lemma under the assumption that $\g$ and $\pa B_{x,r}$ intersect transversally. Indeed, if this is the case then we can find finitely many mutually disjoint closed circular arcs $I_i\subset S^1$, $I_i=[a_i,b_i]$, such that $\g\cap B_{x,r}=\bigcup_i\g((a_i,b_i))$ and $\g\cap\pa B_{x,r}=\bigcup_i\{\g(a_i),\g(b_i)\}$. Arguing by contradiction we may assume that for every $i$ there exists a connected component $A_i$ of $\cl (B_{x,r})\setminus K$ such that $\g(a_i),\g(b_i)\in A_i$. (Note that, possibly, $A_i=A_j$ for some $i\ne j$). By connectedness of $A_i$, for each $i$ we can find a smooth embedding $\tau_i:I_i\rightarrow A_i$ such that $\tau_i(a_i)=\g(a_i)$ and $\tau_i(b_i)=\g(b_i)$; moreover, one can easily achieve this by enforcing $\tau_i(I_i)\cap\tau_j(I_j)=\emptyset$. Finally, we define $\bar\g$ by setting $\bar\g=\g$ on $S^1\setminus\bigcup_iI_i$, and $\bar\g=\tau_i$ on $I_i$.  In this way, $[\bar\g]=[\g]$ in
$\pi_1(\R^{n+1}\setminus H)$, with $\bar\g\cap K\setminus\cl(B_{x,r})=\g\cap K\setminus\cl(B_{x,r})=\emptyset$ and $\bar\g\cap K\cap\cl(B_{x,r})=\emptyset$ by construction; that is, $\bar\g\cap K=\emptyset$. Since there exists $\widetilde\g\in\CC_H$ with $[\widetilde\g]=[\bar\g]=[\g]$ in $\pi_1(\R^{n+1}\setminus H)$ which is uniformly close to $\bar\g$, we entail $\widetilde\g\cap K=\emptyset$, and thus find a contradiction to $K\in\F(H,\CC)$.

  \medskip

  \noindent {\it Step two}: We prove the lemma for any ball $B_{x,r}\subset \mathbb R^{n+1}\setminus H$. Since $\g$ is a smooth embedding, by Sard's theorem we find that $\g$ and $\pa B_{x,s}$ intersect transversally for a.e. $s>0$. In particular, given $\e$ small enough, for any such $s\in(r-\e,r)$ we can construct a smooth diffeomorphism $f_s:\R^{n+1}\rightarrow \R^{n+1}$ such that $f_s=\Id$ on $\R^{n+1}\setminus B_{x,r+2\e}$ and $f_s(y)=x+(r/s)(y-x)$ for $y\in B_{x,r+\e}$, in such a way that
  \begin{equation}
    \label{quaranta}
    \mbox{$f_s\to \Id$ uniformly on $\R^{n+1}$ as $s\to r^-$}\,.
  \end{equation}
  We claim that one can apply step one to $f_s\circ\g$. Indeed, the facts that $f_s\circ\g\in\CC$ and $f_s\circ\g$ and $\pa B_{x,r}$ intersect transversally are straightforward;
  moreover, since $\dist(\g,K\cap\pa B_{x,r})>0$ and by \eqref{quaranta} one easily entails that $(f_s\circ\g)\cap K\setminus B_{x,r}=\emptyset$. Hence, by step one, there exists a proper circular arc $I=[a_s,b_s]\subset S^1$ such that $f_s(\g(a_s))\in A_{i(s)}$ and $f_s(\g(b_s))\in A_{j(s)}$ for $A_i\ne A_j$ connected components of $\cl (B_{x,r})\setminus K$ and $(f_s\circ\g)(a_s,b_s)\subset B_{x,r}$. Up to subsequences, we can assume that $a_s\to \bar{a}$, $b_s\to \bar{b}$ and the arc $[a_s, b_s]$ converges to $[\bar{a}, \bar{b}]$. It follows that $\gamma (\bar{a})$ and $\gamma (\bar{b})$ must be belong to distinct connected components of $\cl (B_{x,r})\setminus K$, otherwise by \eqref{quaranta} $f_s (\gamma (a_s))$ and $f_s (\gamma (b_s))$ would belong to the same connected component for some $s$ close enough to $r$.
By \eqref{quaranta} we also have $\gamma ([\bar{a}, \bar{b}])\subset \cl (B_{x,r})$.

The argument for cubes $Q$ is a routine modification of the one given above and left to the reader.
\end{proof}

\begin{proof}
  [Proof of Theorem \ref{thm plateau}] {\it Step one}: We start showing that $\F(H,\CC)$ is a good class in the sense of Definition \ref{def good class}. To this end, we fix $V\in\F (H,\CC)$ and $x\in V$, and prove that a.e. $r\in (0, \dist (x, H))$ one has $V',V''\in\F(H,\CC)$, where $V'$ is the cone competitor of $V$ in $B_{x,r}$, and $V''$ is a cup competitor of $V$ in $B_{x,r}$. We thus fix $\g\in\CC$ and, without loss of generality, we assume that $\g\cap(V\setminus B_{x,r})=\emptyset$. By Lemma \ref{lemma curve intersezione}, $\gamma$ has an arc contained in $\cl (B_{x,r})$ homeomorphic to $[0,1]$ and whose end-points belong to distinct connected components of $\pa B_{x,r}\setminus V$; we denote by $\sigma :[0,1]\rightarrow \cl(B_{x,r})$ a parametrization of this arc. By construction, either $\s(0)$ or $\s(1)$ must belong to $\g\cap V''\cap \pa B_{x,r}$. This proves that $V''\in\F(H,\CC)$. We now show that $\g\cap V'\cap \cl (B_{x,r})\ne\emptyset$. If $x\in\s$, then, trivially, $V'\cap \sigma\neq \emptyset$; if
$x\not\in\s$, then we can project $\sigma$ radially on $\partial B_{x,r}$, and such projection $\pi\circ\sigma$ must intersect $V'\cap\pa B_{x,r}=V\cap \partial B_{x,r}$ by connectedness. If $z$ is such an intersection point, then
$V'\supset \pi^{-1}(z)\cap \s([0,1])\ne\emptyset$
 as $\pi^{-1}(z)=\lambda z$ for some $\lambda\in(0,1)$. This proves that $V'\in\F(H,\CC)$.

  \medskip

  \noindent {\it Step two}: By step one, given a minimizing sequence $\{K_j\}\subset\F(H,\CC)$ which consists of rectifiable sets, we can find a set $K$ with the properties stated in Theorem \ref{thm generale}. In order to prove the second statement in (b) we just need to show that $K\in\F(H,\CC)$. Suppose by contradiction that some $\gamma\in\CC$ does not intersect $K$. Since both $\gamma$ and $K$ are compact, there exists a positive $\e$ such that the tubular neighborhood $U_{2\e}(\gamma)$ does not intersect $K$ and is contained in $\R^{n+1}\setminus H$. Hence
 $\mu(U_{2\e}(\gamma))=0$, and thus
 \begin{equation}\label{e:contraddice}
 \lim_{j\to\infty}\H^n(K_j\cap U_{\e}(\gamma))=0\,.
 \end{equation}
Observe that there is a diffeomorphism $\Phi: S^1 \times D_{\e} \to U_{\e}(\gamma)$ such that $\Phi|_{S^1\times \{0\}} = \gamma$, where
$D_\rho := \{y\in \mathbb R^n: |y|< \rho\}$. Denote by $\gamma_y$ the parallel curve
$\Phi|_{S^1 \times \{y\}}$. Then $\gamma_y \in [\gamma] \in \pi_1 (\R^{n+1}\setminus H)$ for every $y\in D_{\e}$. Thus we must have
 $K_j \cap (\gamma\times\{y\})\neq \emptyset$ for every $y\in D_{\e}$ and every $j\in \mathbb N$. If we set $\hat\pi:  S^1 \times D_{\e}
\to D_\e$ to be the projection on the second factor and define $\pi: U_{\e}(\gamma)\rightarrow D_\e$ as $\hat{\pi}\circ \Phi^{-1}$,
then $\pi$ is a Lipschitz map. The coarea formula then implies
  \[
  \H^n(K_j\cap U_{\e}(\gamma)) \ge \frac{\omega_n\,\e^n}{(\Lip(\pi))^n}>0\,,
  \]
which contradicts \eqref{e:contraddice}.  This shows that $K\in\F(H,\CC)$, as claimed.

  \medskip

  \noindent {\it Step three}: We show that $K$ is a $(\MM,0,\infty)$-minimal set, i.e.
  \[
  \H^n(K)\le \H^n(\vphi(K))
  \]
  whenever $\vphi:\R^{n+1}\to\R^{n+1}$ is a Lipschitz map such that $\vphi=\Id$ on $\R^{n+1}\setminus B_{x,r}$ and $\vphi(B_{x,r})\subset B_{x,r}$ for some $x\in\R^{n+1}\setminus H$ and $r<\dist(x,H)$. To this end, it suffices to show that given such a function $\vphi$, then $\vphi(K)\in\F (H,\CC)$. We fix $\g\in\CC$ and directly assume that  $\g\cap(K\setminus B_{x,\rho})=\emptyset$ for some $\rho\in(r,\dist(x,H))$. By Lemma \ref{lemma curve intersezione}, there exist two distinct connected components $A$ and $A'$ of $B_{x,\rho}\setminus K$ and a connected component of $\gamma\cap \cl (B_{x,\rho})$ having end-points $p\in\cl(A)\cap\pa B_{x,\rho}$ and $q\in\cl(A')\cap\pa B_{x,\rho}$.
  We complete the proof by showing that $p=\vphi(p)$ and $q=\vphi(q)$ are adherent to distinct connected components of $B_{x,\rho}\setminus\vphi(K)$.
  We argue by contradiction, and denote by $\Om$ the connected component of $B_{x,\rho}\setminus\vphi(K)$ with $p,q\in\cl(\Om)$. If $h$ denotes the restriction of $\vphi$ to $\cl(A)$, then the topological degree of $h$ is defined on $\R^{n+1}\setminus h(\partial A)$, thus in $\Omega$.
  Since $\vphi=\Id$ in a neighborhood of $\pa B_{x,\rho}$, one has $\deg(h,p')=1$ for every $p'$ sufficiently close to $p$; since the degree is locally constant and $\Omega$ is connected, $\deg(h,\cdot)=1$ on $\Omega$. In particular, for every $y\in\Omega$, $\vphi^{-1}(y)\cap A\ne\emptyset$. We apply this with $y=q'$ for some $q'\in\Omega$ sufficiently close to $q$. Let $w\in\vphi^{-1}(q')$: since $\vphi =\Id$ on $\R^{n+1}\setminus B_{x,r}$, if $|q'|>r$ then $w=q'$, and thus $q'\in A$. In other words, every $q'\in B_{x,\rho}$ sufficiently close to $q$ is contained in $A$. We may thus connect in $A$ any pair of points $p',q'\in B_{x,\rho}$ which are sufficiently close to $p$ and $q$ respectively, that is to say, $p$ and $q$ can be connected in $A$. This contradicts $A\ne A'$, and completes the proof of the fact that  $K$ is a $(\MM,0,\infty)$-minimal set. We are thus left to prove (b).

\medskip

\noindent {\it Step four}: We want to show that given $K\in\F(H,\CC)$ with $\H^n(K)<\infty$ there exists $K'\in\F(H,\CC)$  rectifiable such that $\H^n(K')\le\H^n(K)$. The proof is divided in three further steps.
By \cite[2.10.25]{FedererBOOK}, $0=(\om_1\om_n/\om_{n+1})\,\H^{n+1}(K)\ge\int_\R^*\H^n(K\cap\{x_1=t\})\,dt$, thus $\Le^1(\{t\in\R:\H^n(K\cap\{x_1=t\})>0\})=0$. In particular,
\[
\Le^1\Big(\bigcup_{j\in\N}\Big\{t\in(0,1):\H^n\big(K\cap \bigcup_{h\in\Z}\big\{x_1=t+\frac{h}{2^j}\big\}\big)>0\Big\}\Big)=0\,,
\]
so that, for a suitable $x_1^0\in(0,1)$ one has $\H^n(K\cap\{x_1=x_1^0+2^{-j}\,h\})=0$ for every $j\in\N$, $h\in\Z$. This argument can be repeated for each coordinate, so to reach a point $x^0\in \mathbb R^{n+1}$ such that $\H^n (K\cap \{x_m = x^0_m  + 2^{-j}\,h\})=0$ for every $m\in \{1, \ldots , n+1\}$, $j\in\N$, $h\in\Z$.
As a consequence, one finds a grid of open diadic cubes $\Q$ such that $\H^n(K\cap\pa Q)=0$ for every $Q\in\Q$. We let $\W$ be the Whitney's covering of $\R^{n+1}\setminus H$ obtained from $\Q$ as in \cite[Theorem 3, page 16]{steinbook}, so that
if $Q'$ is the concentric cube with twice the size of $Q\in\W$, then $Q'\cap H=\emptyset$.

\medskip

\noindent {\it Step five}: First, for every $Q\in\W$ we define a suitable replacement $K_Q$ in the cube $Q$ such that $K_Q\cap\cl(Q)$ is $\H^n$-rectifiable with $\H^n(K_Q\cap\cl(Q))\le\H^n(K\cap\cl(Q))$ and $K_Q\setminus \cl (Q) = K\setminus \cl (Q)$.
Let us denote by $\{F_i\}_i$ the family of connected components of $Q'\setminus K$ and consider the partitioning problem (into Caccioppoli sets, cf. for instance \cite[Section 4.4]{AFP})
\begin{equation}
  \label{partitioning problem}
  \inf\Big\{\H^n\Big(Q'\cap\bigcup_i\pa^* E_i\Big):\mbox{$\{E_i\}_i$ is a partition modulo $\H^{n+1}$ of $Q'$ with $E_i\setminus Q=F_i\setminus Q$}\Big\}\,.
\end{equation}
Since $F_i$ is open with $\pa F_i\subset K$ and $\H^n(K)<\infty$, the infimum in \eqref{partitioning problem} is finite and there exists a minimizing partition $\{E_i\}_i$ (one can apply, for instance, \cite[Theorem 4.19 \& Remark 4.20]{AFP}). Let the closed set $K_Q$ be given by
\[
K_Q=(K\setminus Q)\cup\Big(\cl(Q)\cap\cl\Big(\bigcup_i\pa^*E_i\Big)\Big)\,.
\]
By a slight modification of \cite[Lemma 30.2]{maggiBOOK}, $\H^n(Q\cap(K_Q\setminus\bigcup_i\pa^*E_i))=0$, so that $\cl(Q)\cap K_Q$ is countably $\H^n$-rectifiable. To prove $\H^n(K_Q\cap\cl(Q))\le\H^n(K\cap\cl(Q))$ it suffices to show
\[
\H^n\Big(\cl(Q)\cap\big(K_Q\setminus\bigcup_i\pa^*E_i\big)\Big)=0\,.
\]
Inasmuch $\H^n(Q\cap(K_Q\setminus\bigcup_i\pa^*E_i))=0$ and $\H^n(K\cap\pa Q)=0$, we just need to prove
\[
\H^n\Big(\pa Q\cap\big((K_Q\setminus K)\setminus\bigcup_i\pa^*E_i\big)\Big)=0\,.
\]
In turn, by \cite[Corollary 6.5]{maggiBOOK}, it is enough to find $c_0>0$ such that
\begin{equation}
  \label{brasil}
  \H^n\big(B_{x,r}\cap \bigcup_i\pa^*E_i\big)\ge c_0\,r^n\,,\qquad\forall x\in\pa Q\cap (K_Q\setminus K)\,,\forall  r<r_x=\dist(x,K\setminus Q)\,.
\end{equation}
We now prove \eqref{brasil}. Let $i_0$ be such that $x\in F_{i_0}$ and, for $r<r_x$, let $G_i=E_i\setminus B_{x,r}$ if $i\ne i_0$, and $G_{i_0}=E_{i_0}\cup B_{x,r}$. Since $\{G_i\}_i$ is admissible in \eqref{partitioning problem}, we find that
\begin{eqnarray*}
f(r):=\H^n\big(\cl(B_{x,r})\cap\bigcup_{i}\pa^*E_i\big)\le \H^n\big(\cl(B_{x,r})\cap\bigcup_{i}\pa^*G_i\big)
=\H^n\big(\pa B_{x,r}\cap\bigcup_{i}\pa^*G_i\big)\,.
\end{eqnarray*}
We next denote by $E_i^{(\tau)}$ the points of $x$ of density $\tau$ of the set $E_i$:
\[
\lim_{r\to 0} \frac{\mathcal{H}^{n+1} (B_{x,r}\cap E_i)}{\omega_{n+1} r^{n+1}} = \tau\, .
\]
Now, for a.e. $r<r_x$, one has $\H^n(\pa B_{x,r}\cap (E_{i_0}^{(0)}\Delta \pa^*G_{i_0}))=0$, as well as
\[
\H^n\big(\pa B_{x,r}\cap (E_i^{(1)}\Delta \pa^*G_i)\big)=0\,,\qquad\forall i\ne i_0\,,\qquad
\H^n\Big(\pa B_{x,r}\cap\big(E_{i_0}^{(0)}\Delta\bigcup_{i\ne i_0}E_i^{(1)}\big)\Big)=0\,.
\]
We thus find that  $f(r)\le \H^n(\pa B_{x,r}\cap E_{i_0}^{(0)})$ for a.e. $r<r_x$; now, again for a.e. $r<r_x$, the set $\pa B_{x,r}\cap E_{i_0}^{(0)}$ has finite perimeter in $\pa B_{x,r}$ with
\[
\H^{n-1}\Big(\pa^*_{\pa B_{x,r}}\big(\pa B_{x,r}\cap E_{i_0}^{(0)}\big)\Delta \big(\pa B_{x,r}\cap\pa^*E_{i_0}\big)\Big)=0\,;
\]
since $\H^n(\pa B_{x,r}\setminus E_{i_0}^{(0)})\ge\H^n(\pa B_{x,r})/2$ by convexity of $Q$, the isoperimetric inequality on $\pa B_{x,r}$ yields $f(r)\le C(n)\H^{n-1}(\pa^*E_{i_0}\cap\pa B_{x,r})^{n/(n-1)}\le C(n)\,f'(r)$ for a.e. $r<r_x$. By arguing as in step one of the proof of Theorem \ref{thm generale}, we complete the proof of \eqref{brasil}.

\medskip

\noindent {\it Step six}: We finally set $K'=\bigcup_{Q\in\W}K_Q\cap\cl(Q)$. By step two, $K'$ is $\H^n$-rectifiable, with
\[
\H^n(K')\le\sum_{Q\in\W}\H^n(K_Q\cap\cl(Q))\le\sum_{Q\in\W}\H^n(K\cap\cl(Q))=\sum_{Q\in\W}\H^n(K\cap Q)\,,
\]
where in the last identity we have used step four. This shows that $\H^n(K')\le\H^n(K)$. We now prove that $K'\in\F(H,\CC)$. Let $\g\in\CC$, so that $\g\cap K\cap\cl(Q)\ne \emptyset$ for some $Q\in\W$. Since $K\cap\pa Q\subset K_Q\cap\pa Q\subset K'\cap\pa Q$, we may directly assume that $\g\cap K\cap Q\ne\emptyset$. By Lemma \ref{lemma curve intersezione}, there exists a connected component $\s$ of $\g\cap\cl(Q)$ with end-points $p\in F_i\cap\pa Q$ and $q\in F_j\cap\pa Q$ for some for some distinct connected components $F_i$ and $F_j$ of $\cl (Q)\setminus K$. If either $p$ or $q$ belongs to $K_Q$ there is nothing to prove; otherwise, $p\in E_i$ and $q\in E_j$. In particular, by connectedness of $\s$, it must be $\s\cap K_Q\cap \cl Q\ne\emptyset$. This completes the proof of (b).
\end{proof}

\section{Proof of Theorem \ref{thm david}}\label{appendix david}

\begin{proof} {\it Step one}: In this and in the next step we prove that $\A (H, K_0)$ is a good class in the sense of Definition \ref{def good class}. Let $K\in\A(H,K_0)$: in this step show \eqref{inf good class} when $L$ is a cup competitor in $B_{x,r}\subset \mathbb R^{n+1}\setminus H$ (at least for a.e. $r$). W.l.o.g. we assume $x=0$ and to simplify the notation we write $B_r$ rather than $B_{0,r}$.
We consider therefore $B_r \cc\R^{n+1}\setminus H$ and assume further $\H^n(K\cap\pa B_r )=0$, which holds for a.e. $r$. Also, for convenience we can rescale and assume $r=1$: we then write $B$ instead of $B_1$. Consider the cup competitor of $K$ in $B$ defined by a given connected component $A$ of $\pa B\setminus K$. Its Hausdorff measure is $\H^n (K\setminus B) + \H^n (\partial B \setminus A)$. Our goal is thus to show that, for any given $\sigma >0$, there is $J\in \mathcal{A} (H, K_0)$ with the property that $J\setminus \cl (B) = K\setminus \cl (B)$ and
$\H^n (J) \leq \H^n (K\setminus B) + \H^n (\partial B \setminus A) + \sigma$, namely
\begin{equation}
  \label{inf good class 2}
  \H^n (J\cap \cl (B)) \leq \H^n (\partial B \setminus A) + \sigma\,
\end{equation}
By definition we need a map $\phi_3\in \Sigma (H)$ such that $J = \phi_3 (K)$. In fact we will build $\phi_3$ so that
$\phi_3 = {\rm Id}$ on $\R^{n+1}\setminus B_{1+\eta}$ for some sufficiently small $\eta$.

$\phi_3$ will be constructed building upon two additional maps $\phi_1$ and $\phi_2$. To construct $\phi_1$ we just fix $x_0\in A$ and a small $\rho$ so that $B_{x_0, \rho} \cap K = \emptyset$. $\phi_1$ then projects $B\setminus B_{x_0, \rho}$ onto $\pa B$ along the rays emanating from $x_0$, while it ``stretches'' $B_{x_0, \rho}\cap \cl (B)$ onto $\cl (B)$. In doing so, we achieve that $K_1=\phi_1(K\cap \cl (B))$ is contained in $\partial B$ and it is disjoint from $B_{x_0,\rho}$.

We next claim the existence of a Lipschitz map $\phi_2:\pa B\to\pa B$ with the property that $\phi_2=\Id$ on $U_\e(K\cap\pa B)$ for some
positive $\e$ and that
\begin{equation}
  \label{phi2}
  \H^n(\phi_2(K_1))\le\H^n(\pa B\setminus A)+\s\,.
\end{equation}
The existence of the map $\phi_2$ will be shown in a moment.

In correspondence of $\e$ we can find $\eta>0$ such that $ B_{1+\eta}\cc\R^{n+1}\setminus H$ and
\[
\frac{K\cap\pa B_{1+t}}{1+t}\subset U_\e(K\cap\pa B)\,,\qquad\forall t\in(0,\eta)\,.
\]
Finally, we define $\phi_3:\R^{n+1}\to\R^{n+1}$ by setting
\[
\phi_3(x)=\left\{
\begin{array}{l l}
 \phi_2\big(\phi_1(x))\,,&\mbox{for $|x|<1$}\,,
 \\
 \frac{|x|-1}\eta\,x+\frac{1+\eta-|x|}\eta\,\phi_2\big(\phi_1(x)\big)\,,&\mbox{for $1\le |x|<1+\eta$}
 \\
 x\,,&\mbox{for $|x|\ge 1+\eta$}\,.
\end{array}
\right .
\]
Notice that $\phi_3$ is a Lipschitz map, with
\[
\phi_3=\Id\quad\mbox{on}\quad (\R^{n+1}\setminus B_{1+\eta})\cup \Big\{(1+t)\,x:t\in(0,\eta)\,x\in U_\e(K\cap\pa B)\Big\}\,.
\]
In particular, $J \setminus \cl (B) = \phi_3(K \setminus \cl (B))=K\setminus \cl (B)$ and $J \cap \cl (B) = \phi_3(K\cap \cl (B))=\phi_2(K_1)$ and, by \eqref{phi2}, \eqref{inf good class 2} holds,

We are thus left to construct the map $\phi_2$.  Up to conjugation with a stereographic projection with pole $x_0$, the existence of $\phi_2$ is reduced to the following problem. Given
\begin{itemize}
\item[(i)] a connected open set $\Omega\subset\R^n$ whose complement is bounded and with $\H^n(\pa\Om)=0$,
\item[(ii)] a ball $B_R\subset \mathbb R^n$ such that $\partial \Omega \subset\subset B_R$
\item[(iii)] and a $\s>0$,
\end{itemize}
find $\e>0$ and a Lipschitz map $\phi:\R^n\to\R^n$ such that
\begin{itemize}
\item[(a)] $\phi=\Id$ on $U_\e(\pa\Om)\cup(\R^n\setminus\Om) \cup \R^n\setminus B_{2R}$
\item[(b)] and $\H^n(\phi(B_{R}\cap \Omega))<\s$.
\end{itemize}
This can be achieved as follows. Let $\W$ be the Whitney decomposition of $B_{2R}\cap \Omega$, constructed from the standard family of diadic cubes in $\mathbb R^n$. Given $\e>0$ we can find a ``face connected'' finite subfamily $\W_0$ of $\W$ such that
\[
(B_R \cap \Om) \setminus U_\e(\pa\Om)\subset \bigcup_{Q\in\W_0}Q\,,
\]
and for which there exists $Q_0\in\W_0$ with $Q_0\setminus B_R\ne\emptyset$. We now construct a Lipschitz map $f:\R^{n+1}\to\R^{n+1}$ such that $f=\Id$ on $\R^{n+1}\setminus \bigcup_{Q\in\W_0}Q$ with
\[
f\Big(\bigcup_{Q\in\W_0}Q \cap B_R\Big)\subset\bigcup_{Q\in\W_0}\pa Q\,.
\]
To this end we choose a ball $U_0\cc Q_0\setminus B_R$, and then define a Lipschitz map $f_0:\R^n\to\R^n$ with $f_0=\Id$ on $\R^n\setminus Q_0$, $f_0(U_0)=Q_0$ and $f_0(Q_0\setminus U_0)=\pa Q_0$ by projecting $Q_0\setminus U_0$ radially from the center of $U_0$ onto $\pa Q_0$, and then by
stretching $U_0$ onto $Q_0$.
Let now $Q_1\in \W_0$ share a hyperface with $Q_0$, so that the side-length of $Q_1$ is at most twice that of $Q_0$. In case the side of $Q_1$ is twice that of $Q_0$, we subdivide $Q_1$ into $2^n$-subcubes and denote by $\hat Q_1$ the one sharing an hyperface with $Q_0$; otherwise we set $\hat Q_1=Q_1$. Let $x_1\in Q_0$ be the reflection of the center of $\hat Q_1$ with respect to the common hyperface between $Q_0$ and $\hat Q_1$. Then we can find a ball $U_1\cc Q_0$ and define a Lipschitz map $\hat f_1:\R^n\to\R^n$ such that $\hat f_1=\Id$ on $\R^n\setminus (Q_0\cup \hat Q_1)$, $\hat f_1((\hat Q_1\cup Q_0)\setminus U_1)\subset\pa(Q_0\cup \hat Q_1)$ and $\hat f_1(
U_1)=\hat Q_1\cup Q_0$. In the case when $\hat Q_1\ne Q_1$ we perform a further radial projection onto $\partial Q_1$ from a small ball centered on the center of $\hat Q_1$. In this way we have constructed a Lipschitz map $f_1:\R^n\to\R^n$ such that $f_1=\Id$ on $\R^n\setminus (Q_0\cup Q_1)$, $f_1((Q_1\cup Q_0)\setminus U_1)\subset\pa Q_0\cup \pa Q_1$ and $f_1(U_1)=Q_1\cup Q_0$. Thus $g_1=f_1\circ f_0$ is a Lipschitz map such that $g_1=\Id$ on $\R^n\setminus (Q_0\cup Q_1)$ and $g_1((Q_0\cup Q_1)\setminus U_0)\subset\pa Q_0\cup\pa Q_1$. A simple iteration concludes the proof.

\medskip

\noindent {\it Step two}: In this step we address cone competitors. As before we consider balls $B_r$ centered at $0$ with $B_r\cc\R^{n+1}\setminus H$. We assume in addition that $K\cap\pa B_r$ is $\H^{n-1}$-rectifiable with $\H^{n-1}(K\cap\pa B_r)<\infty$ and that $r$ is a Lebesgue point of $t\in(0,\infty)\mapsto \H^{n-1}(K\cap\pa B_t)$. All these conditions are fullfilled for a.e. $r$ and again by scaling we can assume that $r=1$ and use $B$ instead of $B_1$. Let $K'$ denote the cone competitor of $K$ in $B$. For $s\in(0,1)$ let us set
\[
\vphi_s(r)=\left\{
\begin{array}{l l}
  0\,,&r\in[0,1-s)\,,
  \\
  \frac{r-(1-s)}s\,,&r\in[1-s,1]\,,
  \\
  r\,,&r\ge 1\,,
\end{array}
\right .
\]
and $\phi_s(x)=\vphi_s(|x|)$ for $x\in\R^{n+1}$. In this way $\phi_s:\R^{n+1}\to\R^{n+1}$ is a Lipschitz map with $\phi_s=\Id$ on $\R^{n+1}\setminus B$. In particular, $\phi_s(K)\setminus B=K\setminus B$ and thus we only need to show that
\[
\limsup_{s\to 0^+}\H^n(\phi_s(K\cap B))\le\H^n(K'\cap B)\,.
\]
Since $\phi_s(K\cap B_{1-s})=\{0\}$, we just have to show that
\[
\limsup_{s\to 0^+}\H^n\big(\phi_s(K)\cap (B\setminus B_{1-s})\big)\le\frac{\H^{n-1}(K\cap\pa B)}n\,.
\]
Denoting by $J^K\phi_s$ the tangential Jacobian of $\phi_s$ with respect to $K$, and letting $I$ be the (at most countable) set of those $t\in(0,1)$ such that $\H^{n-1}(K\cap\pa B_t)>0$, we find
\begin{eqnarray}\nonumber
  &&\H^n\big(\phi_s(K)\cap (B\setminus B_{1-s})\big)=\int_{K\cap(B\setminus B_{1-s})}J^K\phi_s\,d\H^n
  \\\label{ignoriamo}
  &=&\int_{1-s}^1\,dt\int_{K\cap\pa B_t}\frac{J^K\phi_s}{\sqrt{1-(\nu\cdot\hat x)^2}}\,d\H^{n-1}+ \sum_{t\in I\cap(1-s,1)}\left(\textstyle{\frac{t- (1-s)}{1-s}}\right)^n\H^n( K\cap\pa B_t),
\end{eqnarray}
where $\nu(x)\in S^{n+1}\cap (T_xK)^\perp$ for $\H^n$-a.e. $x\in K$ and $\hat x=x/|x|$. We first notice that, for $t\in (1-s,1)$, $\frac{t-(1-s)}{s} \leq 1$. Moreover
\[
\lim_{s\to 0}\sum_{t\in I\cap(1-s,1)}\H^n(K\cap\pa B_t)=0\,,
\]
and thus the second term in \eqref{ignoriamo} can be ignored. At the same time, for a constant $C$,
\[
J^K\phi_s(x)\le C+\sqrt{1-(\hat x\cdot\nu)^2}\,\vphi_s'(|x|)\,\Big(\frac{\vphi_s(|x|)}{|x|}\Big)^{n-1}\,,\qquad\mbox{for $\H^n$-a.e. $x\in K$}\,.
\]
The constant $C$ gives a negligible contribution in the integral as $s\downarrow 0$; as for the second term, having $\vphi_s'=1/s$ on $(1-s,1)$, we find
\[
\int_{1-s}^1 \H^{n-1}(K\cap\pa B_t)\,\vphi_s'(t)\,\Big(\frac{\vphi_s(t)}{t}\Big)^{n-1}\,dt
=\frac1s\int_{1-s}^1 \H^{n-1}(K\cap\pa B_t)\,\Big(\frac{\vphi_s(t)}{t}\Big)^{n-1}\,dt\,.
\]
Since $t=1$ is a Lebesgue point of $t\in(0,\infty)\mapsto \H^{n-1}(K\cap\pa B_t)$, we have
\[
\lim_{s\to 0}\frac1s\int_{1-s}^1 |\H^{n-1}(K\cap\pa B_t)-\H^{n-1}(K\cap\pa B)|\,dt=0\,,
\]
so that, combining the above remarks we find
\[
\limsup_{s\to 0^+}\H^n(\phi_s(K\cap B))\le \H^{n-1}(K\cap\pa B)\,
\limsup_{s\to 0^+}\frac1s\int_{1-s}^1 \Big(\frac{\vphi_s(t)}{t}\Big)^{n-1}\,dt=\frac{\H^{n-1}(K\cap\pa B)}n\,,
\]
as required. This completes the proof that $\mathcal{A} (H, K_0)$ is a good variational class.

\medskip

\noindent {\it Step three}: Having proved the first statement of the theorem, we now show the rest. Under the rectifiability assumption on $K_0$, any minimizing sequence in $\A (H, K_0)$ consists of rectifiable sets and we can therefore apply
Theorem \ref{thm generale}. We thus know that $\H^n\res K_j\weak\mu=\theta\,\H^n\res K$, where $K$ is countably $\H^n$-rectifiable and $\theta\ge 1$. Moreover we assume that $\e_j\downarrow 0$ quantifies the almost minimality of $K_j$, namely $\inf \{\H^n (J) : J\in \A (H, K_0)\} \geq \H^n (K_j) - \e_j$.

In this step we prove that $\theta\le 1$ $\mu$-a.e.. Arguing by contradiction we assume that $\theta(x)=1+\sigma>1$ for some $x$ where $K$ admits an approximate tangent plane $T$ (cf. Step 3 in the Proof of Theorem \ref{thm generale}). W.l.o.g. we can assume $x=0$ and $T = \{y: y_{n+1} =0\}$. By \eqref{e:DLB},
we can find $r_0>0$ such that
  \begin{equation}
    \label{mancoilnome}
      K\cap B_r\subset S_{\e r}\,,\qquad 1+\sigma\le\frac{\mu(\cl(B_r))}{\om_n\,r^n}\le1+\s+ \e\,\s\,, \qquad \forall r<r_0\, ,
  \end{equation}
  where $S_{\e r}=B_r\cap \{|x_{n+1}|<\e r\}$. If we fix any $r<r_0$ we then find $j_0 = j_0 (r)\in\N$ such that
  \begin{equation}
    \label{buonpunto}
      \H^n(K_j\cap B_r)>\Big(1+\frac{\s}2\Big)\,\om_n\,r^n\,,\qquad \H^n( (K_j\cap B_r)\setminus S_{\e r})<\frac\s{4}\,\om_n\,r^n,\qquad\forall j\ge j_0\,,
  \end{equation}
  and thus
  \[
  \H^n(K_j\cap S_{\e r})> \Big(1+\frac{\s}4\Big)\,\om_n\,r^n\,,\qquad\forall j\ge j_0\,.
  \]
  Let us set
  \[
  X_{\e r}=\Big\{x=(x',x_{n+1})\in S_{\e r}:|x'|<(1-\sqrt{\e})\,r\Big\}\,,
  \]
  and define $f:X_{\e r}\cup\ (\R^{n+1}\setminus B_r)\to\R^{n+1}$ with $f(x)=(x',0)$ if $x\in X_{\e r}$ and $f(x) =x$ otherwise. In this way $\Lip (f) \le 1+C\,\sqrt\e$
  and thus by Kirszbraun's theorem (see \cite[2.10.43]{FedererBOOK}) there exists a Lipschitz extension $\hat f:\R^{n+1}\to\R^{n+1}$ with $\Lip\hat f\le 1+C\,\sqrt\e$. Such extension belongs to $\Sigma (H)$ and we thus find
  \begin{eqnarray*}
    \H^n(K_j\cap B_r)-\e_j\le
    \underbrace{\H^n(\hat f(K_j\cap X_{\e r}))}_{I}
    +
    \underbrace{\H^n(\hat f(K_j\cap(S_{\e r}\setminus X_{\e r})))}_{II}
    +
    \underbrace{\H^n(\hat f(K_j\cap (B_r\setminus S_{\e r})))}_{III}\,.
  \end{eqnarray*}
  By construction, $I\le\om_n\,r^n$, while, by \eqref{buonpunto}, $\H^n(K_j\cap B_r)>(1+(\s/2))\om_n\,r^n$ and
  \[
  III\le(\Lip\hat f)^n\,\H^n(K_j\cap (B_r\setminus S_{\e r}))<(1+C\,\sqrt\e)^n\,\frac{\s}{4}\,\om_n\,r^n\,.
  \]
  Hence, as $j\to\infty$,
  \[
  \Big(1+\frac{\s}2\Big)\om_n\,r^n\le\om_n\,r^n+\liminf_{j\to\infty}II+(1+C\,\sqrt\e)^n\,\frac{\s}{4}\,\om_n\,r^n\,,
  \]
  that is,
  \begin{equation}\label{this}
  \Big(\frac12-\frac{(1+C\,\sqrt\e)^n}4\Big)\,\s\le\liminf_{j\to\infty}\frac{II}{\om_n\,r^n}\,.
  \end{equation}
  By \eqref{mancoilnome} and again by the monotonicity of $s^{-n}\,\mu(B_s)$, we finally estimate that
  \begin{eqnarray}\nonumber
    \limsup_{j\to\infty}\,II&\le&(1+C\sqrt{\e})^n\,\mu(\cl(B_r)\setminus B_{(1-\sqrt\e)r})
    \\\label{that}
    &\le&(1+C\sqrt{\e})^n\Big((1+\s+\e\s)-(1+\s)(1-\sqrt\e)^n\Big)\,\om_n\,r^n
  \end{eqnarray}
 However, since $\sigma >0$, \eqref{this} and \eqref{that} are not compatible when $\e$ is sufficiently small.

  \medskip

  \noindent {\it Step four}: We show that $\H^n(K_j)\to \H^n(K)$ and thus the first equality in \eqref{allardo}. We first let $R_0>0$ be such that $H\subset B_{R_0}$ and consider the Lipschitz map $\varphi (x) := \min\{|x|, R_0\} x/|x|$. Obviously $\varphi \in \Sigma (H)$
and we easily compute
\[
\H^n (K_j) - \e \leq \H^n (\varphi (K_j)) \leq \H^n (K_j \cap B_{2R_0}) + \frac{1}{2^n} \H^n (K_j \setminus B_{2R_0})\, .
\]
This implies that $\H^n (K_j \setminus B_{2R_0}) \to 0$.
In order to prove $\H^n(K_j)\to \H^n(K)$, we are left to show that there is no loss of mass at $H$. To this end, let us fix $\eta>0$, and consider $\de>0$ and the map $\pi$ as in \eqref{retraction}. Then, by $\pi\in\S(H)$ and by $\H^n(\pi(U_\de(H)))\le \H^n(H)=0$,
  \begin{eqnarray*}
  \H^n(K)&\le&\limsup_{j\to\infty}\H^n(K_j)\le\limsup_{j\to\infty}\H^n(\pi(K_j))\le (1+\eta)^n\,\limsup_{j\to\infty}\H^n(K_j\setminus U_\de(H))
  \\
  &=&(1+\eta)^n\,\limsup_{j\to\infty}\H^n((K_j\cap \cl(B_{2 R_0}))\setminus U_\de(H))
  \\
  &\le&(1+\eta)^n\,\H^n(K\cap\cl(B_{2 R_0}))\le(1+\eta)^n\,\H^n(K)\,.
  \end{eqnarray*}
The arbitrariness of $\eta$ implies that $\limsup_j \H^n (K_j)  = \H^n (K)$.

  \medskip

 \noindent {\it Step five}: To complete the proof we need to show the second equality in \eqref{allardo}. We argue in two steps, where we borrow some important ideas from \cite{depauwhardt}. We show in this step that $\H^n(K)\le\H^n(\phi(K))$ whenever $\phi\in\S(H)$ is a diffeomorphism. Let $G(n)$ denote the Grassmanian of $n$-planes in $\R^{n+1}$, let ${\rm d}(\tau,\s)$ denote the geodesic distance on $G(n)$, and let $J^\tau\phi$ be the tangential jacobian of $\phi$ with respect to $\tau\in G(n)$. Given $\e>0$ we can find $\de>0$ and a compact set $\hat{K} \subset K$ with $\H^n(K\setminus \hat{K} )<\e$ such that $K$ admits an approximate tangent plane $\tau(x)$ at every $x\in\hat K$,
 \begin{eqnarray}
   \label{approx cont}
    \sup_{x\in \hat{K}}\sup_{y\in B_{x,\de}}|\nabla\phi(x)-\nabla\phi(y)|\le \e\,,\qquad
 \sup_{x\in \hat{K}}\sup_{y\in \hat{K}\cap B_{x,\de}}{\rm d}(\tau(x),\tau(y))<\e\,,
 \end{eqnarray}
 and, moreover, denoting by $S_{x,r}$ the set of points in $B_{x,r}$ at distance at most $\e\,r$ from $x+\tau(x)$, then $K\cap B_{x,r}\subset S_{x,r}$ for every $r<\de$ and $x\in\hat{K}$. By Besicovitch covering theorem we can find a finite disjoint family of closed balls $\{\cl(B_i)\}$ with $B_i=B_{x_i,r_i}\cc\R^{n+1}\setminus H$, $x_i\in \hat{K} $, and $r_i<\de$, such that $\H^n(\hat{K} \setminus\bigcup_iB_i)<\e$. By exploiting the construction of step three, we can find $j(\e)\in\N$ and maps $f_i: \cl (B_i)\to \cl (B_i)$ with $\Lip(f_i)\le 1+C\,\sqrt{\e}$ such that, for a certain $X_i\subset S_i=S_{x_i,\e\,r_i}$,
 \begin{eqnarray}
   \label{ober1}
   &&f_i(X_i)\subset B_i\cap(x_i+\tau(x_i))\,,
   \\
   \label{ober2}
   &&\H^n\Big(f_i\big((K_j\cap B_i)\setminus X_i\big)\Big)\le C\,\sqrt\e\,\om_n\,r_i^n\,,\qquad\forall j\ge j(\e)\,.
 \end{eqnarray}
 By \eqref{approx cont}, \eqref{ober1}, by the area formula, by  $\om_n\,r_i^n\le\H^n(K\cap B_i)$ (thanks to the monotonicity formula), and setting $\a_i=\H^n((K\setminus \hat K)\cap B_i)$,
 \begin{eqnarray}\nonumber
   \H^n(\phi(f_i(K_j\cap X_i)))&=&\int_{f_i(K_j\cap X_i)}J^{\tau(x_i)}\phi(x)\,d\H^n(x)\le(J^{\tau(x_i)}\phi(x_i)+\e)\,\om_n\,r_i^n
   \\\nonumber
   &\le&(J^{\tau(x_i)}\phi(x_i)+\e)\,\H^n(K\cap B_i)
   \le(J^{\tau(x_i)}\phi(x_i)+\e)\,(\H^n(\hat K\cap B_i)+\a_i)
   \\ \nonumber
   &\le&\int_{\hat K\cap B_i}(J^{\tau(x)}\phi(x)+2\e)\,\,d\H^n(x)+((\Lip\phi)^n+\e)\,\a_i
   \\\label{diavolo}
   &=&\H^n(\phi(\hat K\cap B_i))+2\e\,\H^n(K\cap B_i)+((\Lip\phi)^n+\e)\,\a_i\,,
 \end{eqnarray}
 where in the last identity we have used the injectivity of $\phi$. Recalling step three, each map $f_i$ is the identity on $\partial B_i$ Since $\{\cl(B_i)\}$ is a finite disjoint family of closed balls, we can define $f:\R^n\to\R^n$ imposing $f= f_i$ on each $B_i$ and $f=\Id$ on $\R^n\setminus\bigcup_iB_i$. Obviously $f\in\S(H)$. Combining \eqref{ober2} with  $\om_n\,r_i^n\le\H^n(K\cap B_i)$, adding up over $i$, and letting $j\to\infty$ we thus find $\H^n(K_j)-\e_j\le\H^n(\phi(f(K_j)))\le\H^n(\phi(\hat K))+\varrho(\e)$ for every $j\ge j(\e)$, where $\varrho(\e)\to 0$ as $\e\to 0^+$ in a way which depends on $n$, $\Lip\phi$, and $\H^n(K)$ only. We first let $j\to\infty$ and then $\e\to 0$ to prove our claim.

\medskip

\noindent {\it Step six}: By step five, the canonical density one varifold associated to the rectifiable set $K$ turns out to be stationary in $\mathbb R^{n+1}\setminus H$. By Allard's regularity theorem \cite[Chapter 5]{SimonLN} there exists an $\H^n$-negligible closed set $S\subset K$ such that $\Gamma = K \setminus S$ is a real analytic hypersurface. We may now exploit this fact to improve on step five and show that $\H^n(K)\le\H^n(\phi(K))$ for every $\phi\in\S(H)$, showing that $K$ is a sliding minimizer (and hence an $(\MM,0,\infty)$-minimal set). The idea is that, by regularity of $\Gamma$, at a fixed distance from the singular set one can project $K_j$ directly onto $K$, rather than onto its affine tangent planes localized in balls. More precisely, since $\H^n(H\cup S)=0$ and $\H^n(K)<\infty$ one has
  \begin{equation}
    \label{usoquesta}
      \limsup_{j\to\infty}\H^n(K_j\cap U_\de(H\cup S))\le \H^n(K\cap U_\de(H\cup S))=:\varrho(\de)\,,
  \end{equation}
  where $\varrho(\de)\to 0$ as $\de\to 0^+$. If $N_\e(A)$ denotes the normal $\e$-neighborhood upon $A\subset\Gamma$, then, by compactness of $\Gamma_\de=\Gamma\setminus U_\de(H\cup S)$ there exists $\e<\de$ such that projection onto $\Gamma$ defines a smooth map $p:N_{2\e}(\Gamma_\de)\to \Gamma_\de$. We now define a Lipschitz map
\[
f_{\e, \delta} :N_\e(\Gamma_\de)\cup U_{\de/2}(H\cup S) \cup (\R^{n+1}\setminus U_\delta (\Gamma)) \to \R^{n+1}
\]
by setting $f_{\e, \delta}=p$ on $N_\e(\Gamma_\de)$, and $f_{\e, \delta} =\Id$ on the remainder. Observe that
\[
\lim_{\e\downarrow 0} \Lip (f_{\e, \delta}) = 1 < \infty\, .
\]
For every $\delta$ we then choose $\e < \delta$ so that $f = f_{\e, \delta}$ has Lipschitz constant at most $2$ and extend it to a Lipschitz map $\hat{f}$ on $\R^{n+1}$ with the same Lipschitz constant. Obviously $\hat f$ belongs to $\S (H)$. We can then estimate
  \begin{eqnarray}\label{pino}
  \H^n (\hat f(K_j)\setminus \Gamma_\delta)
  \le\,(\Lip\hat f)^n\,\H^n\big(K_j\setminus N_\e (\Gamma_\de)\big)\,.
  \end{eqnarray}
Observe that $\R^{n+1} \setminus N_{\e} (\Gamma_\de)\subset\subset \mathbb R^{n+1}\setminus U_{\e/2} (K) \cup U_{2\de} (H\cup S)$ and thus
  \begin{equation}
    \label{trombone}
      \limsup_{j\to\infty}
    \H^n\big(K_j\setminus N_\e(\Gamma_\de)\big)\le \H^n (K\cap U_{2\de} (H\cap S))\stackrel{\eqref{usoquesta}}{\leq} \varrho (2\de)\, .
  \end{equation}
Combining \eqref{pino} and \eqref{trombone}
\[
\limsup_j \H^n (\hat f (K_j)\setminus \Gamma_\delta)\leq 2^n \varrho (2\de)\, .
\]
On the other hand $\Gamma_\de\subset K$. Thus,
combining \eqref{pino} and \eqref{trombone} with a standard diagonal argument we achieve a sequence of maps $f_j\in \Sigma (H)$ such that $\H^n(f_j (K_j))\setminus K)\to 0$. Since each $K_j$ equals $\psi_j (K_0)$ for some $\psi_j \in \Sigma (H)$, we therefore conclude the existence of a sequence of maps $\{\vphi_j\}\subset\S(H)$ such that $\H^n(\vphi_j(K_0)\setminus K)\to 0$.

We are now ready to show the right identity in \eqref{allardo}. Fix $\phi\in \Sigma (H)$. Then
\begin{align*}
\H^n (\phi (K)) &\geq \liminf_{j\to \infty} \H^n (\phi \circ \varphi_j (K_0))\\
&\geq  \inf \big\{\H^n (J) : J\in \A (H, K_0)\big\} =\H^n(K)\, .
\end{align*}
This shows that $K$ is a sliding minimizer.
\end{proof}

\nocite{Harrison2014}
\nocite{HarrisonPugh12}
\nocite{HarrisonStokes93}

\bibliography{references}
\bibliographystyle{is-alpha}
\end{document}